\documentclass[11pt]{amsart} 

\newcommand{\bbC}{\mathbb{C}}

\newcommand{\bbF}{\mathbb{F}}

\newcommand{\bbR}{\mathbb{R}}

\newcommand{\bbT}{\mathbb{T}}

\newcommand{\bbZ}{\mathbb{Z}}

\usepackage{amsmath, amsxtra, amsthm, amsfonts, amssymb}
\usepackage{mathrsfs}
\usepackage{graphicx}
\usepackage{url}
\usepackage{color}
\usepackage{bbm}
\usepackage{tikz-cd}

\usepackage[T1]{fontenc}
\usepackage{enumitem}
\usepackage{verbatim}
\usepackage{todonotes}

\usepackage[paperheight=11in, 
    paperwidth=8.5in,
    outer=1in,
    inner=1in,
    bottom=1in,
    top=1in]{geometry}

 \usepackage[hidelinks]{hyperref} 
 \hypersetup{
    colorlinks,
    citecolor=magenta,
    filecolor=magenta,
    linkcolor=blue,
    urlcolor=black
}

\usepackage{cleveref}

\newcommand{\Fl}{{\operatorname{Fl}}}

\newcommand{\st}{\operatorname{st}}
\newcommand{\ovl}{\overline}
\DeclareMathOperator{\GrAff}{GrAff}

\newtheorem{thm}{Theorem}[section]
\newtheorem{cor}[thm]{Corollary}
\newtheorem{lem}[thm]{Lemma}
\newtheorem{prop}[thm]{Proposition}

\theoremstyle{definition}
\newtheorem{defn}[thm]{Definition}
\newtheorem{exmp}[thm]{Example}
\newtheorem{rem}[thm]{Remark}

\title{Algebraic and  o-minimal flows beyond the cocompact case}
\author{Spencer Dembner and Hunter Spink} 

\begin{document}
\begin{abstract}
Let $X \subset \mathbb{C}^n$ be an algebraic variety, and let $\Lambda \subset \mathbb{C}^n$ be a discrete subgroup whose real and complex spans agree. We describe the topological closure of the image of $X$ in $\mathbb{C}^n / \Lambda$, thereby extending a result of Peterzil-Starchenko in the case when $\Lambda$ is cocompact. 

We also obtain a similar extension when $X\subset \mathbb{R}^n$ is definable in an o-minimal structure with no restrictions on $\Lambda$, and as an application prove the following conjecture of Gallinaro: for a closed semi-algebraic $X\subset \mathbb{C}^n$ (such as a complex algebraic variety) and $\exp:\mathbb{C}^n\to (\mathbb{C}^*)^n$ the coordinate-wise exponential map, we have $\overline{\exp(X)}=\exp(X)\cup \bigcup_{i=1}^m \exp(C_i)\cdot \mathbb{T}_i$ where $\mathbb{T}_i\subset (\mathbb{C}^*)^n$ are positive-dimensional compact real tori and $C_i\subset \mathbb{C}^n$ are semi-algebraic.
\end{abstract}
\maketitle
\section{Introduction}

Let $V$ be a vector space over $\bbR$ or $\bbC$, $\Lambda \subset V$ a discrete subgroup, and $X \subset V$ a closed subset. Let $\pi \colon V \to V / \Lambda$ denote the projection map. Ullmo-Yafaev \cite{Ullmo} observed that as $R \to \infty$, the closures 
$$ \ovl{\pi(\{x \in X: \| x \| \ge R \})} $$
form what we may call a ``flow'' of sets, which
converge to a closed limiting set $\Fl(X)$ exactly containing those cosets $v+\Lambda$ which an unbounded sequence of points $x_i \in X$ can approach in Euclidean distance. Evidently we have
$$ \ovl{\pi(X)} = \pi(X) \cup \Fl(X).$$


In the case where $X \subset \bbC^n$ is a complex curve and $\Lambda \subset \bbC^n$ is cocompact, \cite[Theorem 2.4]{Ullmo} shows $\Fl(X)$ consists of finitely many translates of real subtori $\bbT \subset \bbC^n / \Lambda$. This was later extended to arbitrary discrete subgroups by Dinh-Vu \cite{DinhVu}. 


When $X$ is a higher-dimensional variety, an example of Peterzil-Starchenko \cite[Section 8]{PSTori} shows this description of $\Fl(X)$ doesn't hold, even if $\Lambda$ is cocompact. However, they show that in this cocompact case, replacing translated subtori by algebraic families of translated subtori is enough \cite[Theorem 1.1]{PSTori}. 

Our main result shows the cocompactness restriction can be removed when the real and complex spans of $\Lambda$ agree. For $\bbF=\bbR$ or $\bbC$ and $S$ a subset of an $\bbF$-vector space $V$, let $\bbF S$ denote its $\bbF$-linear span. 


\begin{thm}\label{thm:complex}
Let $X\subset \mathbb{C}^n$ be an algebraic variety, $L\subset \mathbb{C}^n$ a $\mathbb{C}$-subspace,  $\Lambda\subset \mathbb{C}^n$ a discrete subgroup with $\mathbb{R}\Lambda=\mathbb{C}\Lambda=L$, and $\pi:\mathbb{C}^n\to \mathbb{C}^n/\Lambda$ the quotient map. Then there are algebraic varieties $C_1,\ldots,C_m\subset \mathbb{C}^n$ with $\dim_{\mathbb{C}}C_i<\dim_{\mathbb{C}}X$ and positive-dimensional compact real tori $\mathbb{T}_1,\ldots,\mathbb{T}_m\subset L/\Lambda$ such that
$$\Fl(X)=\bigcup_{i=1}^m (\pi(C_i)+\mathbb{T}_i).$$
In particular, $\overline{\pi(X)}=\pi(X)\cup \bigcup_{i=1}^m (\pi(C_i)+\mathbb{T}_i)$.
\end{thm}
\begin{rem}
As in \cite{PSTori}, we will be able to take $C_1,\ldots,C_m$ to only depend on $X$ and $L$, and take positive-dimensional $\mathbb{C}$-subspaces $V_1,\ldots,V_m\subset L$ only depending on $X$ and $L$ such that $\mathbb{T}_i=\overline{\pi(V_i)}$.
\end{rem}
When $\Lambda$ is cocompact, \cite{PSTori} also shows that we may assume $C_i$ is a finite collection of points if $\mathbb{T}_i$ is inclusion-maximal among $\mathbb{T}_1,\ldots,\mathbb{T}_m$, but this is false in this more general setting even for $X=\mathbb{C}^2$ and $\Lambda=(\mathbb{Z}\oplus i\mathbb{Z})\times \{0\}$.

The condition that the $\mathbb{R}$-span and $\mathbb{C}$-span of $\Lambda$ are equal is essential, even without the dimension conditions, as the following example of \cite{DinhVu} shows.
\begin{exmp}[{\cite[{Example in Section 4}]{DinhVu}}]\label{exmp: unequal spans}
Let $X$ be the hypersurface $z=xye^{-i\pi/4}(y+1)$ in $\mathbb{C}^3$, and let $\Lambda=\mathbb{Z}^3$. Letting $p \colon \bbC \to \bbC / \bbZ$ be the projection and defining
$$L=\{s\in \mathbb{C}:\arg(s^2+s)=\pi/4\text{ or }5 \pi /4\}\cup \{s\in \mathbb{C}:s^2+s=0\},$$
we have
$$\Fl(X)=[(\mathbb{C}/\mathbb{Z})\times p(L) \times (\mathbb{C}/\mathbb{Z}) ]\cup[p(e^{i\pi/4}\mathbb{R})\times (\mathbb{C}/\mathbb{Z})\times (\mathbb{C}/\mathbb{Z})].$$
But one can show that any translate of a proper subset $C+V\subset \mathbb{C}^3$ with $C$ an algebraic variety and $V$ a real vector subspace cannot intersect  $\mathbb{C}\times L \times \mathbb{C}$
in a full-dimensional set. So if $\overline{\pi(X)}=\pi(X)\cup \bigcup_{i=1}^m(\pi(C_i)+\mathbb{T}_i)$, we obtain a contradiction by applying the Baire category theorem to the countable union  $\pi^{-1}(\overline{\pi(X)})=\bigcup_i C_i+\pi^{-1}(\mathbb{T}_i)$ of translates for $i=1,\ldots,m$ of proper subsets of $\mathbb{C}^3$ of the form $C_i+V_i$, where $V_i$ is the connected component of $\pi^{-1}(\mathbb{T}_i)$ at $0$, obtains a contradiction.
\end{exmp}



Gallinaro \cite[Question 6.3.2]{Gallinaro}, in work on Zilber's exponential-algebraic closedness conjecture (an equivalent formulation of Zilber's quasi-minimality conjecture \cite{ZilberQuasiminimal} as established by Bays-Kirby \cite[Theorem 1.5]{BK}), has conjectured that such a decomposition should be possible if the $C_i$ are taken semi-algebraic. Our second main result, an o-minimal version of \Cref{thm:complex} with no restrictions on $\Lambda$ generalizing \cite[Theorem 1.2]{PSTori}, confirms this conjecture. 

\begin{thm}\label{thm:real}
Fix an o-minimal structure $\mathbb{R}_{om}$ extending the field $\mathbb{R}$. Let $X\subset \mathbb{R}^n$ be a definable closed set, $L\subset \mathbb{R}^n$ an $\mathbb{R}$-subspace, and $\Lambda\subset \mathbb{R}^n$ a discrete subgroup with $\mathbb{R}\Lambda=L$. Then for  $\pi:\mathbb{R}^n\to \mathbb{R}^n/\Lambda$ the quotient map, the following is true. There are definable closed sets $C_1,\ldots,C_m\subset \mathbb{R}^n$ with $\dim C_i<\dim X$ and compact positive-dimensional real tori $\mathbb{T}_1,\ldots,\mathbb{T}_m\subset L/\Lambda$ such that
$$\Fl(X)=\bigcup_{i=1}^m (\pi(C_i)+\mathbb{T}_i).$$
In particular, $\overline{\pi(X)}=\pi(X)\cup \bigcup_{i=1}^m (\pi(C_i)+\mathbb{T}_i)$.
\end{thm}
\begin{rem}
As in \cite{PSTori}, we will be able to take $C_1,\ldots,C_m$ to only depend on $X$ and $L$, and take positive-dimensional $\mathbb{R}$-subspaces $V_1,\ldots,V_m\subset L$ only depending on $X$ and $L$ such that $\mathbb{T}_i=\overline{\pi(V_i)}$.
\end{rem}

In the cocompact case \cite{PSTori} also shows that we may take $C_i$ to be compact if $\mathbb{T}_i$ is inclusion-maximal amongst $\mathbb{T}_1,\ldots,\mathbb{T}_m$, but this is false in this more general setting even for $X=\mathbb{R}^2$ and $\Lambda=\mathbb{Z}\times \{0\}$.
\begin{cor}[{\cite[Question 6.3.2]{Gallinaro}}]
Let $X\subset \mathbb{C}^n$ be a closed semi-algebraic set (such as a complex algebraic variety), and $\exp:\mathbb{C}^n\to (\mathbb{C}^*)^n$ the map $(z_1,\ldots,z_n)\mapsto (\exp(z_1),\ldots,\exp(z_n))$. Then there are semi-algebraic subsets $C_1,\ldots,C_m\subset \mathbb{C}^n$ with $\dim C_i<\dim X$ and compact positive-dimensional real tori $\mathbb{T}_1,\ldots,\mathbb{T}_m\subset (\mathbb{C}^*)^n$ such that 
$$\overline{\exp(X)}=\exp(X)\cup \bigcup_{i=1}^m \exp(C_i)\cdot \mathbb{T}_i.$$
\end{cor}
\begin{proof}
We apply the theorem with $\Lambda=\mathbb{Z}^n\subset \mathbb{C}^n$ in the semi-algebraic o-minimal structure $\mathbb{R}_{alg}$, since $\exp$ is surjective with  kernel $\mathbb{Z}^n$, so $\exp$ is identified with the quotient map $\mathbb{C}^n\to \mathbb{C}^n/\mathbb{Z}^n$.
\end{proof}
To prove the cocompact case, the main technical tool introduced in \cite{PSTori} is a definable collection of ``asymptotic flats'' $\mathcal{A}^{\mathbb{R}}(X)$ (or $\mathcal{A}^{\mathbb{C}}(X)$) associated to $X$ which encode the linear behaviour of sequences of points of $X$. Roughly speaking, limiting sequences of points are encoded by realizations of $X$ in a saturated extension $\mathcal{R}$ of $\mathbb{R}$ (or in $\mathcal{C}:=\mathcal{R}\oplus i\mathcal{R}$), and we associate to such a point the smallest flat defined over $\mathbb{R}$ (or $\mathbb{C}$) which is infinitesimally close to this realization.

The key idea in our proof is that if we decompose our ambient vector space as $L\oplus L^{\perp}$, then a sequence of points $x_1,x_2,\ldots$ limiting to a coset of $\Lambda$ has convergent $L^{\perp}$-component, and the corresponding asymptotic flat would then have linear part contained in $L\oplus \{0\}$. This motivates us to consider just the sub-family of flats of $\mathcal{A}(X)$ with no linear part in the $L^{\perp}$ direction, and by doing so we can carry out a very similar argument in both the complex algebraic and o-minimal settings.

\noindent\textbf{Acknowledgements.} We thank Y. Peterzil, S. Starchenko, and F. Gallinaro for helpful answers to questions. 

\section{Model theory setup}
We closely follow the setup of \cite{PSTori}, which we recall briefly. Fix a cardinal $\kappa > 2^{\omega}$. $\mathbb{R}$ will be considered as an ordered field over the language $\langle 0,1,+,\times,<\rangle$. Let $\mathbb{R}_{full}$ be the expansion of the field $\mathbb{R}$ by all subsets of $\mathbb{R}^n$ for all $n$, $\mathcal{L}_{full}$ the corresponding language, and $\mathcal{R}_{full}$ a $\kappa$-saturated extension. Denote by $\mathcal{R}$ the underlying field. Fix an o-minimal expansion $\mathbb{R}_{om}$ of $\mathbb{R}$ with corresponding language $\mathcal{L}_{om}$, and let $\mathcal{R}_{om}$ be the corresponding reduct $\mathcal{R}_{full}\upharpoonright \mathcal{L}_{om}$ of $\mathcal{R}_{full}$, which is also $\kappa$-saturated.

We denote $\mu_{\mathcal{R}}\subset \mathcal{R}$ for the set of infinitesimal elements and $\mathcal{O}_{\mathcal{R}}=\mathcal{R}\oplus \mu_{\mathcal{R}}$ the set of bounded elements (neither of which are definable in $\mathcal{R}_{om}$). We denote $\st:\mathcal{O}_{\mathcal{R}}\to \mathbb{R}$ the standard part map, and for a subset $S\subset \mathcal{R}^n$ we denote
$$\st(S):=\st(S\cap \mathcal{O}_{\mathcal{R}}^n).$$

The complex numbers $\mathbb{C}=\mathbb{R}\oplus i\mathbb{R}$ will be considered as a structure with signature $\langle 0,1,+,\times\rangle$. We denote by $\mathcal{C}=\mathcal{R}\oplus i\mathcal{R}$, the set of infinitesimals $\mu_{\mathcal{C}}=\mu_{\mathcal{R}}\oplus i\mu_{\mathcal{R}}$, and the set of bounded elements $\mathcal{O}_{\mathcal{C}}=\mathcal{O}_{\mathcal{R}}\oplus i\mathcal{O}_{\mathcal{R}}$. With these identifications, we may similarly define $\st:\mathcal{O}_{\mathcal{C}}\to \mathcal{C}$ and $\st(S)=\st(S\cap \mathcal{O}_{\mathcal{C}}^n)$ for a subset $S\subset \mathcal{C}^n$. Through the identification $\mathcal{C}=\mathcal{R}\oplus i\mathcal{R}$, we may consider $\mathcal{C}_{full}$ the expansion of $\mathbb{C}$ by all subsets of $\mathbb{C}^n$ for all $n$.

Denote $\mathcal{L}_{val}=\langle 0,1, +, \cdot, \mathcal{O}\rangle$ the language of a valued field, with $\mathcal{O}$ a unary predicate for the valuation ring, and $\mathcal{C}_{val}$ the $\mathcal{L}_{val}$-structure of the valued field $\mathcal{C}$ augmented with the set $\mathcal{O}_{\mathcal{C}}$ of bounded elements. Note also that $\mu_{\mathcal{C}}$ is definable in $\mathcal{C}_{val}$ by the formula $ x=0\wedge x^{-1}\not\in \mathcal{O}_{\mathcal{C}}$. Note that $\mathcal{C}_{val}$ is not a reduct of $\mathcal{C}_{full}$ and in particular not $\kappa$-saturated, since for example the finitely-satisfiable collection of $2^\omega$ formulas $x\in \mathcal{O}_{\mathcal{C}} \cup \bigcup_{c\in \mathbb{C}} \{x\not\in c+\mu_{\mathcal{C}}\}$ is not satisfiable.

We will denote $(\mathbb{F},\mathcal{F})$ to be either $(\mathbb{R},\mathcal{R})$ or $(\mathbb{C},\mathcal{C})$.

If $\mathcal{L}$ is one of the languages $\mathcal{L}_{om},\mathcal{L}_{full}$ when working in the real setting or $\mathcal{L}_{val},\mathcal{L}_{full}$ when working in the complex setting, then an $\mathcal{L}$-definable partial type $p(x_1,\ldots,x_n)$ is a finitely satisfiable collection of formulas over the language $\mathcal{L}$ with free variables $x_1,\ldots,x_n$ and constants in $A$. We denote by $p(\mathcal{F})\subset \mathcal{F}^n$ the set of realizations of $p$ over $\mathcal{F}$.

Suppose we are in one of the $\kappa$-saturated structures $\mathcal{R}_{om},\mathcal{R}_{full}$, or $\mathcal{C}_{full}$, and consider types $p$ defined over a set of constants $A$ with $|A|<\kappa$. Two such types $p,q$ are considered equivalent if for every finite $p_0\subset p$ there exists a finite $q_0\subset q$ such that $q_0(\mathcal{F})\subset p_0(\mathcal{F})$ and vice versa. By $\kappa$-saturation, two such types $p,q$ are equivalent if and only if $p(\mathcal{F})=q(\mathcal{F})$. If $p,q$ are any two types, then the sum $p+q$ is defined by
$$p+q=\{\exists y,z\in \mathcal{F}\text{ with } x=y+z\text{ and }f(y)\text{ and } g(z)\text{ both true}\}_{f\in p,g\in q}.$$
By $\kappa$-saturation, for two such types we have $(p+q)(\mathcal{F})=p(\mathcal{F})+q(\mathcal{F})$, and consequently for three such types $p,q,r$, the types $(p+q)+r$ and $p+(q+r)$ are equivalent.

For $S\subset \mathcal{F}^n$, we can consider $S$ as an $\mathcal{F}_{full}$-definable subset of $\mathcal{F}^n$. By abuse of notation we also use $S$ to denote the partial type of all $\mathcal{F}_{full}$-formulas with constants in $\mathbb{F}$ satisfied by $S$, and we write $S^{\sharp}:=S(\mathcal{F})$ for the set of realizations of $S$ over $\mathcal{F}$. Also by abuse of notation, we identify $\mu_{\mathcal{R}}$ with the partial type $\{-r<x<r\}_{r\in \mathbb{R}_{>0}}$ in $\mathcal{R}_{om}$ or $\mathcal{R}_{full}$ and $\mu_{\mathcal{C}}$ with the partial type $\{-r<Re(x),Im(x)<r\}_{r\in \mathbb{R}_{>0}}$ in $\mathcal{C}_{full}$, so that  $\mu_{\mathcal{R}}(\mathcal{R})=\mu_{\mathcal{R}}$ and $\mu_{\mathcal{C}}(\mathcal{C})=\mu_{\mathcal{C}}$. These abuses of notation are consistent with the sum notation for types, as by $\kappa$-saturatedness for $S,T\subset \mathcal{F}^n$ we have
\begin{align*}(S+T)^{\sharp}=S^{\sharp}+T^{\sharp},\text{ and }
(S+\mu_{\mathcal{C}}^n)(\mathcal{F})=S^{\sharp}+\mu_{\mathcal{C}}^n.
\end{align*}
Finally, by $\kappa$-saturatedness, for any $S\subset \mathbb{F}^n$ we have
$$\st(S^{\sharp})=\overline{S},$$
where $\overline{S}\subset \mathbb{F}^n$ is the topological closure.
\section{Flows via non-standard realizations}
Note that
$$\pi^{-1}\Fl(X)=\bigcap_{R\in \mathbb{R}_{>0}}\overline{(X-B(0,R))+\Lambda}.$$
We will give an alternate description of $\pi^{-1}\Fl(X)$ via non-standard realizations.
\begin{defn}
Let $X\subset \mathbb{F}^n$ be an $\bbR_{om}$-definable set for $\bbF = \bbR$, or a constructible set for $\bbF = \bbC$. We define the partial type
$$X_{\infty}^{\mathbb{F}}=\begin{cases}\{x\in X\text{ and }\|x\|\ge R\}_{R\in \mathbb{R}_{>0}}& \mathbb{F}=\mathbb{R}\\
\{x\in X\text{ and }x\not\in \mathcal{O}_{\mathcal{C}}^n\}&\mathbb{F}=\mathbb{C}.\end{cases}$$
This is either an $\mathcal{R}_{om}-$partial type with constants in $\mathbb{R}$ or a $\mathcal{C}_{val}$-partial type with constants in $\mathbb{C}$.
\end{defn}
\begin{thm}
For $X\subset \mathbb{F}^n$ as above, we have $X_{\infty}^{\mathbb{F}}(\mathcal{F})=X^{\sharp}\setminus \mathcal{O}_{\mathcal{F}}^n$, and
$$\pi^{-1}\Fl(X)=\st(X_{\infty}^{\mathbb{F}}(\mathcal{F})+\Lambda^{\sharp}).$$
\end{thm}
\begin{proof}
That $X_{\infty}^{\mathbb{F}}(\mathcal{F})=X^{\sharp}\setminus \mathcal{O}_{\mathcal{F}}^n$ is trivial when $\mathbb{F}=\mathbb{C}$, and when $\mathbb{F}=\mathbb{R}$ this follows from the fact that $\mathcal{O}_{\mathcal{R}}^n$ is by definition the set of bounded elements of $\mathcal{R}^n$.

Now, for $\mathbb{F}=\mathbb{C}$, since by treating $\mathbb{C}^n$ as a real vector space and $X$ as a semi-algebraic (and in particular $\mathbb{R}_{om}$-definable) set in $\mathbb{R}^{2n}$ we have $X_{\infty}^{\mathbb{C}}(\mathcal{C})=X_{\infty}^{\mathbb{R}}(\mathcal{R})$ by the previous paragraph, and $\st$ is unchanged under restriction of scalars, it suffices to prove the statement for $\mathcal{F}=\mathcal{R}$.

If $\pi(z)\in Fl(X)$, the type $(X_{\infty}^{\mathbb{R}}+\Lambda)\cup (z+\mu_{\mathcal{R}}^n)$ is finitely satisfied.
Thus, by saturation of $\mathcal{R}$, there exists $z'\in (X_{\infty}^{\mathbb{R}}+\Lambda)(\mathcal{R})=X_{\infty}^{\mathbb{R}}(\mathcal{R})+\Lambda^{\sharp}$ infinitesimally close to $z$. Since $\st$ is insensitive to addition by an element of infinitesimals, we have $z\in \st(X_{\infty}^{\mathbb{R}}(\mathcal{R})+\Lambda^{\sharp})$. Conversely, if $\alpha\in (X_{\infty}^{\mathbb{R}}(\mathcal{R})+\Lambda^{\sharp})\cap \mathcal{O}_{\mathcal{F}}^n$, then since $X_{\infty}^{\mathbb{R}}(\mathcal{R})\subset (X-B(0,R))^{\sharp}$, we have $\st(\alpha)\in \st((X-B(0,R))^{\sharp}+\Lambda^{\sharp})=\overline{(X-B(0,R))+\Lambda}$ for all $R\in \mathbb{R}_{>0}$.
\end{proof}

\section{Asymptotic flats}
Let
$\GrAff^{\mathbb{F}}_{i}(\mathbb{F}^n)$ be the Grassmannian of affine $i$-dimensional flats of $\mathbb{F}^n$, and let $$\GrAff^{\mathbb{F}}(\mathbb{F}^n)=\GrAff^{\mathbb{F}}_{0}(\mathbb{F}^n)\sqcup \ldots \sqcup \GrAff^{\mathbb{F}}_{n}(\mathbb{F}^n).$$
Note that
$\GrAff^{\mathbb{F}}_{0}(\mathbb{F}^n)=\mathbb{F}^n$ and $\GrAff^{\mathbb{F}}_{n}(\mathbb{F}^n)=\{\{\mathbb{F}^n\}\}$.

For $\alpha\in \mathcal{F}^n$, we denote by $A_{\alpha}^{\mathbb{F}}$ the minimal $\mathbb{F}$-flat under inclusion such that $\alpha\in (A_\alpha^\mathbb{F})^\#+\mu_{\mathcal{F}}^n$ (that this exists is \cite[Proposition 4.2]{PSTori} when $\mathbb{F}=\mathbb{C}$ and as noted in \cite[Section 7.1]{PSTori}, the proof also works for $\mathbb{F}=\mathbb{R}$).
For $0 \le i \le n$ define
$\mathcal{A}_{i}^{\mathbb{F}}(X):=\{A_{\alpha}^{\mathbb{F}}:\alpha\in X^{\sharp}\text{ and } \dim_{\mathbb{F}}(A_{\alpha}^{\mathbb{F}})=i\}\subset {\GrAff}_{i}^{\mathbb{F}}(\mathbb{F}^n)$
and
$$\mathcal{A}^{\mathbb{F}}(X):=\{A_{\alpha}^{\mathbb{F}}:\alpha\in X^{\sharp}\}=\bigsqcup_{i=0}^n A_{\alpha,i}^{\mathbb{F}}(X)\subset \GrAff^{\mathbb{F}}_n.$$
By \cite[Theorems 4.5 and 7.1]{PSTori} this is a definable subset of $\GrAff^\mathbb{F}(\mathbb F^n)$, and by definition  $A_{\alpha}^{\mathbb{F}}\in \mathcal{A}_0^{\mathbb{F}}(X)$ if and only if $\alpha\in \mathcal{O}^n$, in which case $A_{\alpha}^{\mathbb{F}}=\st(\alpha)$ (see \cite[Lemma 4.8]{PSTori}). In particular $\mathcal{A}_{0}^{\mathbb{F}}(X)=\overline{X}$. Write $\mathcal{A}_{pos}^{\mathbb{F}}(X):=\bigsqcup_{i=1}^n \mathcal{A}^{\mathbb{F}}_{i}(X)$.

A slight refinement of the arguments in \cite{PSTori} shows that if $\Lambda$ is cocompact in $\mathbb{F}^n$, then $$\Fl(X)=\bigcup_{A\in \mathcal{A}_{pos}^{\mathbb{F}}(X)}\overline{\pi(A)}.$$
We will adapt this now to lattices $\Lambda$ with $\mathbb{R}\Lambda=L$ an $\mathbb{F}$-linear subspace of $\mathbb{F}^n$.

Let $\GrAff_i^{\mathbb{F}}(\mathbb{F}^n;L)\subset \GrAff_i^{\mathbb{F}}(\mathbb{F}^n)$ be the subset of affine $i$-dimensional flats $A$ of $\mathbb{F}^n$ such that the linear part of $A$ is contained in $L$, and $$\GrAff^{\mathbb{F}}(\mathbb{F}^n;L)=\GrAff_0^{\mathbb{F}}(\mathbb{F}^n;L)\sqcup \ldots \sqcup \GrAff^{\mathbb{F}}_{\dim_{\mathbb{F}}(L)}(\mathbb{F}^n;L)$$

For $0\le i \le \dim_{\mathbb{F}}(L)$, we define
$\mathcal{A}^{\mathbb{F}}_i(X;L):=\mathcal{A}^{\mathbb{F}}_i(X)\cap \GrAff_i^{\mathbb{F}}(\mathbb{F}^n;L)$
and
$$\mathcal{A}^{\mathbb{F}}(X;L):=\mathcal{A}^{\mathbb{F}}(X)\cap \GrAff^{\mathbb{F}}(\mathbb{F}^n;L).$$
Clearly this is a definable subset, and $\mathcal{A}^{\mathbb{F}}(X;L)=X$. Write $\mathcal{A}_{pos}^{\mathbb{F}}(X;L):=\mathcal{A}^{\mathbb{F}}_{pos}(X)\cap \GrAff^{\mathbb{F}}(\mathbb{F}^n;L)$.

\begin{thm}
\label{thm:Fl}
If $\Lambda\subset \mathbb{F}^n$ is a discrete subgroup and $L\subset \mathbb{F}^n$ is an $\mathbb{F}$-linear subspace such that $\mathbb{R}\Lambda=L$, then
$$\Fl(X)=\bigcup_{A\in \mathcal{A}^{\mathbb{F}}_{pos}(X;L)}\overline{\pi(A)}.$$
\end{thm}
Before proving the theorem, we start by proving some lemmas.

\begin{lem}
\label{lem:Aposdescription}
Suppose that $L=\mathbb{F}^k\times \{0\}^{n-k}$. Then for $\alpha\in \mathcal{F}^n$, we have $A_{\alpha}^{\mathbb{F}}\in \GrAff_i(\mathbb{F}^n;L)$ if and only if $\alpha=(\alpha_1,\alpha_2)\in \mathcal{F}^k\times \mathcal{O}_{\mathcal{F}}^{n-k}$, in which case $A_{\alpha}^{\mathbb{F}}\subset \mathbb{F}^k\times \{\st(\alpha_2)\}$. In particular,
$$\mathcal{A}_{pos}^{\mathbb{F}}(X;L)=\{A_{\alpha}^{\mathbb{F}}:\alpha\in X_{\infty}^{\mathbb{F}}(\mathcal{F})\cap (\mathcal{F}^k\times \mathcal{O}_{\mathcal{F}}^{n-k})\}.$$
\end{lem}
\begin{proof}
If $A_{\alpha}^{\mathbb{F}}\in \GrAff_i(\mathbb{F}^n;L)$, then $A_{\alpha}^{\mathbb{F}}\subset \mathbb{F}^k\times \{a\}$ for some $a\in \mathbb{F}^{n-k}$, so $$\alpha\in (A_{\alpha}^{\mathbb{F}})^{\sharp}+\mu_\mathcal{F}^n\subset (\mathbb{F}^k\times \{a\})^{\sharp}+\mu_{\mathcal{F}}^n= \mathcal{F}^k\times \{a\}+\mu_\mathcal{F}^n\subset \mathcal{F}^k\times \mathcal{O}_{\mathcal{F}}^{n-k}.$$
Conversely, if $\alpha\in \mathcal{F}^k\times \mathcal{O}_{\mathcal{F}}^{n-k}$, then  $$\alpha\in \mathcal{F}^k\times \{\st(\alpha_2)\}+\mu_{\mathcal{F}}^n= (\mathbb{F}^k\times \{\st(\alpha_2)\})^{\sharp}+\mu_{\mathcal{F}}^n,$$ so $A_{\alpha}^{\mathbb{F}}\subset \mathbb{F}^k\times \{\st(\alpha_2)\}=L+(0,\st(\alpha_2))$ and this shows $A_{\alpha}^{\mathbb{F}}\in \GrAff_i(\mathbb{F}^n;L)$.
\end{proof}
\begin{defn}
For $\alpha\in \mathcal{F}^n$, we define $p_{\alpha}$ to be the complete $\mathcal{C}_{val}$-type of $\alpha$ with $\mathbb{C}$-constants if $\mathbb{F}=\mathbb{C}$ and to be the complete $\mathcal{R}_{om}$-type of $\alpha$ with $\mathbb{R}$-constants if $\mathbb{F}=\mathbb{R}$. 
\end{defn}
\begin{lem}
\label{lem:palphaproperty}
For $\alpha=(\alpha_1,\alpha_2)\in \mathcal{F}^k\times \mathcal{O}_{\mathcal{F}}^{n-k}$ and $a=\st(\alpha_2)$, then $p_{\alpha}(\mathcal{F})\subset \mathcal{F}^k\times \{a\}+\mu_{\mathcal{F}}^{n}$. Also, if $\alpha\in X_{\infty}(\mathcal{F})$ then $p_{\alpha}(\mathcal{F})\subset X_{\infty}(\mathcal{F})$, so in particular
$$\alpha\in X_{\infty}^{\mathbb{F}}(\mathcal{F})\cap (\mathcal{F}^k\times \mathcal{O}_{\mathcal{F}}^{n-k}) \implies p_{\alpha}(\mathcal{F})\subset X_{\infty}^{\mathbb{F}}(\mathcal{F})\cap (\mathcal{F}^k\times \mathcal{O}_{\mathcal{F}}^{n-k}).$$
\end{lem}
\begin{proof}
$\mathcal{F}^k\times \{a\}+\mu_{\mathcal{F}}^n=(\mathbb{F}^k\times \{a\})^{\sharp}+\mu_{\mathcal{F}}^n$ corresponds to a partial type that $\alpha$ satisfies by construction. Similarly for  $X_{\infty}^{\mathbb{F}}$ (noting that this partial type is definable over $\mathcal{C}_{val}$ with $\mathbb{C}$-coefficients when $\mathbb{F}=\mathbb{C}$ and over $\mathbb{R}_{om}$ with $\mathbb{R}$-coefficients when $\mathbb{F}=\mathbb{R}$).
\end{proof}
\begin{lem}
For $A_{\alpha}^{\mathbb{F}}\in \GrAff(\mathbb{F}^n;L)$ we have $\overline{A_{\alpha}^{\mathbb{F}}+\Lambda}=\st(p_{\alpha}(\mathcal{F})+\Lambda^{\sharp})$.
\end{lem}
\begin{proof}
After a linear change of coordinates, we may assume that $\Lambda\subset L=\mathbb{F}^k\times \{0\}$. Then by \Cref{lem:Aposdescription} we have $\alpha=(\alpha_1,\alpha_2)\in \mathcal{F}^k\times \mathcal{O}_{\mathcal{F}}^{n-k}$, and writing $a=\st(\alpha_2)$ we have $A_{\alpha}^{\mathbb{F}}\subset \mathbb{F}^k\times \{a\}$, so the left hand side is contained in $\mathbb{F}^k\times \{a\}$. Additionally, $\Lambda^{\sharp}\subset L^{\sharp}=\mathcal{F}^k\times \{0\}^{n-k}$ and by \Cref{lem:palphaproperty} we also have  $p_{\alpha}(\mathcal{F})\subset  \mathcal{F}^k\times \{a\}+\mu_{\mathcal{F}}^{n}$, so the right hand side is also contained in $\mathbb{F}^k\times \{a\}$.

Let $\Lambda'\subset \{0\}^k\times \mathbb{F}^{n-k}$ be a discrete subgroup with $\mathbb{R}\Lambda'=\{0\}^k\times \mathbb{F}^{n-k}$. Then $\mathbb{R}(\Lambda+\Lambda')=\mathbb{F}^n$ so by \cite[Proposition 5.1]{PSTori} in the complex case and \cite[Proposition 7.3]{PSTori} in the real case, we have
$$\overline{A_{\alpha}^{\mathbb{F}}+\Lambda}+\Lambda'=\overline{A_{\alpha}^{\mathbb{F}}+\Lambda+\Lambda'}=\st(p_{\alpha}(\mathcal{F})+(\Lambda+\Lambda')^{\sharp})=\st(p_{\alpha}(\mathcal{F})+\Lambda^{\sharp}+(\Lambda')^{\sharp}).$$
Because $p_{\alpha}(\mathcal{F}),\Lambda^{\sharp}\subset \mathcal{F}^k\times \mathcal{O}_{\mathcal{F}}^{n-k}$, for the last $n-k$ coordinates of the sum $p_{\alpha}(\mathcal{F})+\Lambda^{\sharp}+(\Lambda')^{\sharp}$ to be bounded we need the element from $(\Lambda')^{\sharp}$ to lie in $\{0\}\times\mathcal{O}_{\mathcal{F}}^{n-k}$, which in particular means it is infinitesimally close to an element of $\Lambda'$. Therefore,
$$\overline{A_{\alpha}^{\mathbb{F}}+\Lambda}+\Lambda'=\st(p_{\alpha}(\mathcal{F})+\Lambda^{\sharp})+\Lambda'.$$
Intersecting both sides of the equality with $\mathbb{F}^k\times \{a\}$, we obtain the desired equality.
\end{proof}
\begin{proof}[Proof of \Cref{thm:Fl}]
We may suppose that $\Lambda\subset L=\mathbb{F}^k\times \{0\}^{n-k}$. Then since $\Lambda^{\sharp}\subset \mathcal{F}^k\times \{0\}^{n-k}$, if a sum $\alpha+\lambda\in \mathcal{O}_{\mathcal{F}}^n$ for some $\lambda\in \Lambda^{\sharp}$ then we must have $\alpha\in \mathcal{F}^k\times \mathcal{O}_{\mathcal{F}}^{n-k}$. Therefore, 
$$\pi^{-1}(\Fl(X))=\bigcup_{\alpha\in X_{\infty}^{\mathbb{F}}(\mathcal{F})}\st(\alpha+\Lambda^{\sharp})=\bigcup_{\alpha\in X_{\infty}^{\mathbb{F}}(\mathcal{F})\cap(\mathcal{F}^k\times \mathcal{O}_{\mathcal{F}}^{n-k})}\st(\alpha+\Lambda^{\sharp})$$
$$=\bigcup_{\alpha\in X_{\infty}^{\mathbb{F}}(\mathcal{F})\cap (\mathcal{F}^k\times \mathcal{O}_{\mathcal{F}}^{n-k})}\st(p_{\alpha}(\mathcal{F})+\Lambda^{\sharp})=\bigcup_{A_{\alpha}^{\mathbb{F}}\in \mathcal{A}^{\mathbb{F}}_{pos}(X;L)}\overline{A_{\alpha}^{\mathbb{F}}+\Lambda}$$
where the third equality is by \Cref{lem:palphaproperty}.
\end{proof}

\begin{rem}
    In the cocompact case, every positive-dimensional algebraic set has a nontrivial asymptotic flat by  \cite[Lemma 4.8]{PSTori}. In contrast, $\mathcal{A}^{\mathbb{F}}_{pos}(X;L)$ may be empty. For example, the parabola $y = x^2$ in $\bbR^2$ has no nontrivial asymptotic flats relative to the lattice $\Lambda = \bbZ \times 0$, because any unbounded solution is unbounded in both coordinates. 
\end{rem}

\section{Neat families}
In this section we take ``definable'' to mean defianble either over $\bbR_{om}$ in the language of ordered fields, or over $\bbC$ in the language of fields, according to whether the setting is real or complex. By Chevalley's theorem, in the latter case the definable sets are exactly constructible sets, i.e., Boolean combinations of algebraic sets.

For a flat $A\in \GrAff^{\mathbb{F}}(\mathbb{F}^n)$, we write $L(A)\subset \mathbb{F}^n$ for the linear part of $A$ (the $\mathbb{F}$-subspace such that we can write $A=L(A)+c$ for some $c$), and for a subset $T\subset \GrAff^{\mathbb{F}}(\mathbb{F}^n)$, we define $L(T)\subset \mathbb{F}^n$ to be the smallest $\mathbb{F}$-linear space containing $L(A)$ for all $A\in T$.
\begin{defn}
A ``neat'' family of flats is a definable family $T\subset \GrAff_i^{\mathbb{F}}(\mathbb{F}^n)$ such that
\begin{enumerate}
    \item $T$ is a connected $\mathbb{R}$-submanifold.
    \item For any nonempty open subset $U\subset T$ we have $L(U)=L(T)$.
\end{enumerate}
\end{defn}
\begin{thm}[{\cite[Theorem 7.7 and Section 6.1]{PSTori}}]
Every definable family $T\subset \GrAff^{\mathbb{F}}(\mathbb{F}^n)$ can be written as a finite union of neat definable families.
\end{thm}
\begin{proof}
In the real setting, this is exactly what is proved in \cite[Theorem 7.7]{PSTori}. In the complex setting, we may decompose $T$ into finitely many smooth pieces, and then further decompose each of those pieces into its finitely many irreducible components, so it suffices to show that if $T$ is smooth and irreducible then it is neat. A smooth irreducible set is a connected $\mathbb{R}$-submanifold, so it suffices to show that for any nonempty open subset $U\subset T$ that $L(U)=L(T)$. Any nonempty open subset $U\subset T$ is Zariski dense in $T$ by irreducibility of $T$, so because the Zariski-closed set of those $A \in T$ such that $L(A) \subset L(U)$ contains $U$, it must therefore be all of $T$. Hence $L(A) \subset L(U)$ for all $A \in T$, so $L(T) \subset L(U)$, and $L(U) = L(T)$.
\end{proof}
By this theorem, we may break up the definable family of flats as $\mathcal{A}_{pos}^{\mathbb{F}}(X;L)=T_1\cup \ldots \cup T_m$ where each $T_i$ is a neat family in $\GrAff_j(\mathbb{F}^n;L)$ with $1\le j \le \dim(L)$. We need the following proposition, the analogue of \cite[Proposition 6.2 and Proposition 7.6]{PSTori} for $\GrAff_j(\mathbb{F}^n;L)$.
\begin{prop}
\label{prop:dense}
If $T$ is a neat family of $j$-dimensional flats of $\mathbb{F}^n$ in $\GrAff_j(\mathbb{F}^n;L)$, then $\bigcup_{A\in T}\pi(A)$ is topologically dense in $\bigcup_{A\in T}\pi(A+L(T))$.
\end{prop}
\begin{proof}
The proof of \cite[Proposition 7.6]{PSTori} carries over with essentially no modifications.

First we show the analogue of \cite[Proposition 7.6(1)]{PSTori}, that the set $\{A\in T: \overline{\pi(L(A))}=\overline{\pi(L(T))}\}$ is topologically dense in $T$. Indeed, there are only countably many closed real Lie subgroups of $L/\Gamma$, so by the Baire Category theorem it suffices to show for any proper real Lie subgroup $\mathbb{T}\subsetneq \overline{\pi(L(T))}$ that the set $\{A\in T:\pi(L(A))\subset \mathbb{T}\}$ is nowhere dense in $T$. This set is equal to $\{A\in T: L(A)\subset W\}$, where $W$ is the maximal $\mathbb{F}$-linear subspace contained in $\pi^{-1}(\mathbb{T})$ (either the connected component $W'$ of $0$ in the real setting or $W'\cap iW'$ in the complex setting). Since this is definable, it suffices to show that it does not contain an open subset $U$ of $T$. But since $T$ is neat, $L(U)=L(T)$, so not every  $L(A)$ for $A\in U$ can be contained in $W$ since $W$ is a proper subspace of $L(T)$.

Now we show $\bigcup_{A\in T}\pi(A)$ is topologically dense in $\bigcup_{A\in T}\pi(A+L(T))$, the analogue of \cite[Proposition 7.6(2)]{PSTori}. Indeed, by what we have just shown, there is a sequence $A_1,A_2,\ldots \to A$ such that $\overline{\pi(L(A_i))}=\overline{\pi(L(T))}$, or equivalently $\overline{\pi(A_i)}=\overline{\pi(A_i+L(T))}$. Hence $\bigcup_i \overline{\pi(A_i)}=\bigcup_i \overline{\pi(A_i+L(T))}$ contains $\overline{\pi(A+L(T))}$ in its closure, and therefore $\bigcup_i \pi(A)$ contains $\overline{\pi(A+L(T))}$ in its closure.
\end{proof}
Now the proofs of \Cref{thm:complex} and \Cref{thm:real} proceed identical to \cite[Theorem 7.8]{PSTori}. We sketch how this works. For each $i$, let $L(T_i)^{\perp}$ be a complementary $\mathbb{F}$-subspace to $L(T_i)$. For $A\in T_i$ the intersection $(A+L(T_i))\cap L(T_i)^{\perp}$ is a single point $c_{A,i}$, and we let $C_i'=\bigcup_{A\in T_i}c_{i,A}\subset L(T_i)^{\perp}$ be the closure of the definable set of such points as we range over $A\in T$. Letting $C_i=\overline{C_i'}$, we have $\dim_{\mathbb{F}}C_i=\dim_{\mathbb{F}}C_i'$, and \cite[Section 6.2 and 7.3.1, Proofs of Clause (i)]{PSTori} applied to $C_i'$ shows that with this construction, $\dim_{\mathbb{F}}C_i=\dim_{\mathbb{F}}C_i'<\dim_{\mathbb{F}} X$.  Let $V_i=L(T_i)$ and $\mathbb{T}_i=\overline{\pi(V_i)}$. Then we have (by \Cref{thm:Fl} for the first equality, by \Cref{prop:dense} and the fact that $\Fl(X)$ is closed for the second equality, that $A+L(T_i)=c_{A,i}+V_i$ for the third equality, and that $\Fl(X)$ is closed for the fourth equality) that
$$\Fl(X)=\bigcup_i\bigcup_{A\in T_i}\overline{\pi(A)}=\bigcup_i\bigcup_{A\in T_i} \overline{\pi(A+L(T_i))}=\bigcup_i\bigcup_{A\in T_i}\pi(c_{A,i})+\overline{\pi(V_i)}=\bigcup_i \pi(C_i)+ \mathbb{T}_i.$$

\bibliographystyle{plain}
\bibliography{biblio}
\end{document}